\documentclass[12pt]{article}

\makeatletter
\def\input@path{{styles/}}
\makeatother

\newcommand{\UsePackage}[1]{%
  \IfFileExists{styles/#1.sty}{%
      \usepackage{styles/#1}%
   }{%
      \IfFileExists{../styles/#1.sty}{%
         \usepackage{../styles/#1}%
      }{%
         \usepackage{#1}%
      }%
   }%
}

\usepackage{natbib}
\usepackage[T1]{fontenc}
\usepackage{lmodern}
\usepackage{textcomp}

\usepackage{amsmath}%
\usepackage{amssymb}%
\usepackage[table]{xcolor}%

\usepackage{todonotes}
\usepackage[in]{fullpage}%

\usepackage[amsmath,thmmarks]{ntheorem}%
\theoremseparator{.}%

\usepackage{titlesec}%
\titlelabel{\thetitle. }%
\usepackage{xcolor}%
\usepackage{mleftright}%
\usepackage{xspace}%
\usepackage{graphicx}
\usepackage{hyperref}%
\usepackage[parfill]{parskip}
\usepackage{bbm}

\usepackage{hyperref}%
\hypersetup{%
      unicode,
      breaklinks,%
      colorlinks=true,%
      urlcolor=[rgb]{0.25,0.0,0.0},%
      linkcolor=[rgb]{0.5,0.0,0.0},%
      citecolor=[rgb]{0,0.2,0.445},%
      filecolor=[rgb]{0,0,0.4},
      anchorcolor=[rgb]={0.0,0.1,0.2}%
}
\usepackage[ocgcolorlinks]{ocgx2}

\theoremseparator{.}%

\theoremstyle{plain}%
\newtheorem{theorem}{Theorem}

\newtheorem{lemma}{Lemma}

\theoremstyle{plain}%
\theoremheaderfont{\sf} \theorembodyfont{\upshape}%
\newtheorem*{remark:unnumbered}[theorem]{Remark}%
\newtheorem{remark}[section]{Remark}%

\newcommand{\myqedsymbol}{\rule{2mm}{2mm}}

\theoremheaderfont{\em}%
\theorembodyfont{\upshape}%
\theoremstyle{nonumberplain}%
\theoremseparator{}%
\theoremsymbol{\myqedsymbol}%
\newtheorem{proof}{Proof:}%

\providecommand{\emphind}[1]{}%
\renewcommand{\emphind}[1]{\emph{#1}\index{#1}}

\definecolor{blue25emph}{rgb}{0, 0, 11}

\providecommand{\emphic}[2]{}
\renewcommand{\emphic}[2]{\textcolor{blue25emph}{%
      \textbf{\emph{#1}}}\index{#2}}

\providecommand{\emphi}[1]{}%
\renewcommand{\emphi}[1]{\emphic{#1}{#1}}

\definecolor{almostblack}{rgb}{0, 0, 0.3}

\providecommand{\emphw}[1]{}%
\renewcommand{\emphw}[1]{{\textcolor{almostblack}{\emph{#1}}}}%

\providecommand{\emphOnly}[1]{}%
\renewcommand{\emphOnly}[1]{\emph{\textcolor{blue25}{\textbf{#1}}}}

\newcommand{\HLink}[2]{\hyperref[#2]{#1~\ref*{#2}}}
\newcommand{\HLinkSuffix}[3]{\hyperref[#2]{#1\ref*{#2}{#3}}}

\providecommand{\deflab}[1]{}
\renewcommand{\deflab}[1]{\label{def:#1}}

\providecommand{\eqlab}[1]{}%
\renewcommand{\eqlab}[1]{\label{equation:#1}}

\newcommand{\remove}[1]{}%

\usepackage[inline]{enumitem}

\newlist{compactenumA}{enumerate}{5}%
\setlist[compactenumA]{topsep=0pt,itemsep=-1ex,partopsep=1ex,parsep=1ex,%
   label=(\Alph*)}%

\newlist{compactenuma}{enumerate}{5}%
\setlist[compactenuma]{topsep=0pt,itemsep=-1ex,partopsep=1ex,parsep=1ex,%
   label=(\alph*)}%

\newlist{compactenumI}{enumerate}{5}%
\setlist[compactenumI]{topsep=0pt,itemsep=-1ex,partopsep=1ex,parsep=1ex,%
   label=(\Roman*)}%

\newlist{compactenumi}{enumerate}{5}%
\setlist[compactenumi]{topsep=0pt,itemsep=-1ex,partopsep=1ex,parsep=1ex,%
   label=(\roman*)}%

\newlist{compactitem}{itemize}{5}%
\setlist[compactitem]{topsep=0pt,itemsep=-1ex,partopsep=1ex,parsep=1ex,%
   label=\ensuremath{\bullet}}%

\providecommand{\BibLatexMode}[1]{}
\providecommand{\BibTexMode}[1]{}

\ifx\UseBibLatex\undefined%
  \renewcommand{\BibLatexMode}[1]{}
  \renewcommand{\BibTexMode}[1]{#1}
\else
  \renewcommand{\BibLatexMode}[1]{#1}
  \renewcommand{\BibTexMode}[1]{}
\fi

\BibLatexMode{%
   \usepackage[bibencoding=utf8,style=alphabetic,backend=biber]{biblatex}%
   \UsePackage{my_biblatex}%
}

\numberwithin{figure}{section}%
\numberwithin{table}{section}%

\usepackage{tikz}
\newcommand{\KL}{\textsf{KL}}
\newcommand{\comp}{\mathsf{c}}

\BibLatexMode{%
   \bibliography{template}
}

\begin{document}

\title{Tight Bounds on the Binomial CDF, and the Minimum of i.i.d Binomials, in terms of KL-Divergence}

\author{Xiaohan Zhu\thanks{Department of Statistics, The University of Chicago} \and 
Mesrob I. Ohannessian\thanks{Department of Electrical and Computer Engineering, The University of Illinois at Chicago} \and Nathan Srebro\thanks{Toyota Technological Institute at Chicago}}

\date{}
\maketitle

\begin{abstract}
    We provide finite sample upper and lower bounds on the Binomial tail probability which are a direct application of Sanov's theorem. We then use these to obtain high probability upper and lower bounds on the minimum of i.i.d.~Binomial random variables.  Both bounds are finite sample, asymptotically tight, and expressed in terms of the KL-divergence.
\end{abstract}

The purpose of this note is to provide, in a self-contained and concise way, both upper {\em and} lower bounds on the Binomial tail, and through that, on the minimum of i.i.d.~Binomial random variables.  The upper bound on the minimum of i.i.d.~Binomials can be seen as a special case of known uniform concentration guarantees.  Our main purpose here, beyond providing an explicit, simple, and self-contained presentation, is to provide also simple matching {\em lower} bounds.  Such lower bounds can then be used in constructing and analyzing lower bounds for different learning methods and situations. 

We first prove a finite sample bound on the Binomial cumulative distribution function (CDF) by using Sanov's theorem. Throughout, all logarithms are base-$2$ and KL divergence is measured in bits.
\begin{lemma}[Binomial tail]
\label{lemma}
    Let $X \sim \frac{1}{n}\textnormal{Bin}(n, p)$ be a scaled Binomial random variable. Then for $a \leq p$,
    \[
    \log \mathbb{P}(X \leq a) \in -n \KL(a\Vert p) \pm \left(4\log(n+1) + \left[\log \frac{p}{1-p}\right]_+\right),
    \]   where $\KL(\alpha\Vert \beta)$ denotes $\KL(Ber(\alpha)\Vert Ber(\beta)) = \alpha\log\frac{\alpha}{\beta} + (1-\alpha)\log\frac{1-\alpha}{1-\beta}$.
\end{lemma}

\begin{proof}
    We write $X = \frac{1}{n}\sum_{i = 1}^n X_i$, where $X_i \overset{\mathrm{iid}}{\sim} \textnormal{Ber}(p)$, and so $X_1, X_2, \cdots, X_n$ is a sequence of $n$ symbols from the alphabet $\mathcal{X} = \{0,1\}$ with type $(1-X,X)$. Denote the true distribution $Q=\textnormal{Ber}(p)$. 
    
    The upper bound follows directly from Sanov's theorem \citep{TC}:
    \begin{equation}\label{eqn1}
        \log \mathbb{P}(X \leq a) \leq -n \KL(a\Vert p) + 2\log(n+1).
    \end{equation}
    To get a finite sample lower bound, we round $a$ to a multiple of $1/n$.  That is, let $k=\lfloor a n \rfloor$ and $\tilde{a}=k/n$, so that $a-1/n < \tilde{a} \leq a$.
    
    Let $\mathcal{P}_n = \{(P(0), P(1)): (\frac{0}{n}, \frac{n}{n}),(\frac{1}{n}, \frac{n-1}{n}), \cdots, (\frac{n}{n}, \frac{0}{n}) \}$ be the set of types with denominator $n$, and $E = \{P: P(1) \leq a\}$. Then the type $P_{\tilde{a}} = (1-\tilde{a}, \tilde{a})$ lies in the intersection $E \cap \mathcal{P}_n$.

   Given the type $P \in \mathcal{P}_n$, let $T(P) = \{x \in \mathcal{X}^n: P_x = P\}$ denote the type class of $P$, which is the set of sequences of length $n$ and and type $P$. Then, by adapting equations (11.104) to (11.106) in the lower bound proof of \cite{TC}, we have:
    
    \begin{equation*}
         \begin{split}
        \mathbb{P}(X \leq a) = Q^n(E) &= \sum_{P \in E \cap \mathcal{P}_n} Q^n\left( T \left( P \right) \right)\\
        &\geq Q^n \left(T\left(P_{\tilde{a}}\right)\right)\\
        &\geq \frac{1}{(n+1)^2} 2^{-n\KL(\tilde{a}\Vert p)}.
    \end{split}
    \end{equation*}

    Taking the logarithm on both sides yields:
    \begin{equation}
     \log \mathbb{P}(X \leq a) \geq -2\log(n+1) - n\KL(\tilde{a}\Vert p). \tag{*} \label{*}
    \end{equation}
     
    

        Since $a - \tilde{a} < 1/n$, $H(a) - H(\tilde{a}) < H(\frac{1}{n}) < \frac{2}{n} \log n$. This implies that $\KL(\tilde a\Vert p) - \KL(a\Vert p)= (a-\tilde{a})\log \frac{p}{1-p} + H(a) - H(\tilde{a}) \leq \frac{1}{n} \left[ \log \frac{p}{1-p} \right]_+ + \frac{2}{n} \log n$. Plugging this in the inequality \eqref{*} yields
    \begin{align}
        \log \mathbb{P}(X \leq a) &\geq -2\log(n+1) - n\KL(\tilde{a}\Vert p)\notag\\
        &\geq -2\log(n+1) - \left(n\KL(a\Vert p) + 2\log n + \left[\log \frac{p}{1-p}\right]_+ \right)\notag\\
        &\geq -n \KL(a\Vert p) - 4\log(n+1) - \left[\log \frac{p}{1-p}\right]_+.\label{eq:cdflower}
    \end{align}
    
    The upper bound \eqref{eqn1} and lower bound \eqref{eq:cdflower} together yield the desired result.
\end{proof}

Next, we use the finite sample bound on the Binomial CDF to prove the following concentration bounds of the minimum of i.i.d Binomials in terms of KL divergence.

\begin{theorem}[minimum of i.i.d Binomial]\label{thm1}
    Let $\{X_i\}_{i = 1}^r \overset{\mathrm{iid}}{\sim} \frac{1}{n}\textnormal{Bin}(n, p)$, $Z = \min_{i = 1, \cdots, r} X_i$. Given fixed confidence parameter $\delta \in (0,1)$, let $\Delta(\delta, p, n) = \log \frac{1}{\delta/2} + 4\log (n+1) + \left[ \log \frac{p}{1-p} \right]_+ $. If $\Delta(\delta, p, n) < \log r$, then with probability $1-\delta$, we have
    \[
    Z < p \text{, and } \KL(Z\Vert p) \in \frac{\log r \pm \Delta(\delta, p, n)}{n},
    \]
    except that if $\KL(0 \Vert p) < \frac{\log r -\Delta(\delta, p, n)}{n}$, then with probability $1-\delta$, $Z = 0$.
\end{theorem}

\begin{proof}
    Consider any interval $[a,b]$, such that $a\leq b< p$. Define the following events:
    \begin{eqnarray*}
        U &=& \{ \KL(Z\Vert p) \leq \KL(a\Vert p) \}, \\
        L &=& \{ \KL(Z\Vert p) \geq \KL(b\Vert p) \}, \\
        A &=& \{ Z \geq a \},\textrm{ and} \\
        B &=& \{ Z \leq b \}.
    \end{eqnarray*}
    \begin{center}
       \begin{tikzpicture}
            \draw[black, thick] (0,0) -- (10,0);
            \draw[thick] (2,0.2) -- (2,-0.2) node[below] {$a$};
            \draw[thick] (5,0.2) -- (5,-0.2) node[below] {$b$};
            \draw[thick] (7,0.2) -- (7,-0.2) node[below] {$p$};
            \draw[magenta, thick] (0,-1) -- (5,-1) node[midway, above] {$B$};
            \draw[cyan, thick] (2,-1.5) -- (10,-1.5) node[midway, above] {$A$};
            \draw[blue, thick] (2,-2) -- (9,-2) node[midway, above] {$U$};
            \draw[red, thick] (0,-2.5) -- (5,-2.5);
            \draw[red, thick] (8,-2.5) -- (10,-2.5);
            \node[above] at (4,-2.5) {\textcolor{red}{$L$}};
        \end{tikzpicture}    
    \end{center}
    By the monotonicity of the KL divergence, we have that $B \subseteq L$ and $A \cap B \subseteq U$ (but note that we generally \emph{don't} have $A \subseteq U$). This means that $A \cap B \subseteq U \cap L$, and consequently:
    \[
        \mathbb{P}(U \cap L) \geq  \mathbb{P}(A \cap B)  = 1 - \mathbb{P}(A^\comp) - \mathbb{P}(B^\comp).
    \]
    The theorem will follow from choices of $a$ and $b$ that help bound $\mathbb{P}(A^\comp)$ and $\mathbb{P}(B^\comp)$.
    
    Using the fact that $a<p$, along with the union bound and Lemma \ref{lemma}, we have 
    \[
    \mathbb{P}(A^\comp) =\mathbb{P}(Z < a) \leq \mathbb{P}(Z \leq a) \leq r \cdot\mathbb{P}(X_1 \leq a) \leq r \cdot 2^{-n\KL(a\Vert p) + 4\log (n+1) + \left[ \log \frac{p}{1-p} \right]_+}.
    \]
    Suppose $\KL(0\Vert p) \geq \frac{\log r + \Delta(\delta, p, n)}{n}$. Since $\KL(p\Vert p) = 0$, and KL is continuous by its first argument, by intermediate value theorem, we can choose $0 \leq a <p$ such that
    \begin{equation} \label{eq:KLap} 
    \begin{split}
    \KL(a\Vert p)  &= \frac{\log r + \Delta(\delta, p, n)}{n}\\
    &= \frac{\log r + \log \frac{1}{\delta/2} + 4\log (n+1) + \left[ \log \frac{p}{1-p} \right]_+}{n},
    \end{split}
    \end{equation}
     which gives $r \cdot 2^{-n\KL(a\Vert p) + 4\log (n+1) + \left[ \log \frac{p}{1-p} \right]_+}=\delta/2$.
    Thus, by choosing $0 \leq a < p$ according to \eqref{eq:KLap}, we get $\mathbb{P}(A^\comp) \leq \frac{\delta}{2}$.
    
    If $\KL(0\Vert p) < \frac{\log r + \Delta(\delta, p, n)}{n}$, in other words, there is no $0 \leq a < p$ satisfying \eqref{eq:KLap}, then take $a = 0$. And in this case, the upper bound of the theorem trivially holds for any $Z < p$ because
    \begin{align*}
        \mathbb{P}\left(\KL(b \Vert p) \leq \KL(Z\Vert p) \leq \KL(0 \Vert p) < \frac{\log r + \Delta(\delta, p, n)}{n}\right) &\geq \mathbb{P}(0 \leq Z \leq b)
        = 1 - \mathbb{P}(Z > b).
    \end{align*}

    On the other hand, by the independence of data points, we have:
    \begin{equation}\label{eq:indep}
    \mathbb{P}(B^\comp) = \mathbb{P}(Z > b) = (1-\mathbb{P}(X_1 \leq b))^r.
    \end{equation}
    Using the inequality $\forall x \in [0,1], k > 0: (1-x)^k \leq e^{-kx}$ and Lemma \ref{lemma}, we have
    \begin{equation}\label{eq:expbd}
    (1-\mathbb{P}(X_1 \leq b))^r \leq \exp\left(-r \cdot \mathbb{P}(X_1 \leq b)\right) \leq \exp \left(-r \cdot 2^{-n\KL(b\Vert p) - 4\log (n+1) - \left[ \log \frac{p}{1-p} \right]_+}\right).
    \end{equation}
    Suppose $\KL(0\Vert p) \geq \frac{\log r - \log \ln \frac{1}{\delta/2} - 4\log (n+1) - \left[ \log \frac{p}{1-p} \right]_+}{n}$, again by the intermediate value theorem, we can choose $0 \leq b < p$ such that 
    \begin{equation} \label{eq:KLbp}
    \KL(b\Vert p) = \frac{\log r - \log \ln \frac{1}{\delta/2} - 4\log (n+1) - \left[ \log \frac{p}{1-p} \right]_+}{n}, 
    \end{equation}
    which gives $\exp \left(-r \cdot 2^{-n\KL(b\Vert p) - 4\log (n+1) - \left[ \log \frac{p}{1-p} \right]_+}\right) = \delta/2$.
    Thus, by choosing $0 \leq b <p$ according to \eqref{eq:KLbp}, we get $\mathbb{P}(B^\comp) \leq \frac{\delta}{2}$.

    If $\KL(0\Vert p) < \frac{\log r - \log \ln \frac{1}{\delta/2} - 4\log (n+1) - \left[ \log \frac{p}{1-p} \right]_+}{n}$, in other words, there is no $0\leq b < p$ satisfying \eqref{eq:KLbp}, then by combining \eqref{eq:indep} and \eqref{eq:expbd},
    \[
        \mathbb{P}(Z > 0)\leq \exp \left(-r \cdot 2^{-n\KL(0\Vert p) - 4\log (n+1) - \left[ \log \frac{p}{1-p} \right]_+}\right)
        \leq \frac{\delta}{2}.
    \]
    So in this case, we have with probability $\geq \frac{\delta}{2} > 1 - \delta$, $Z = 0$.

    Therefore, by choosing $a$ and $b$ as above, we get
    \begin{equation*}
        \mathbb{P}\Big(\KL(Z\Vert p) \in \big(\KL(b\Vert p), \KL(a\Vert p)\big)\Big) = \mathbb{P}(U\cap L) \geq 1 - \mathbb{P}(A^\comp) - \mathbb{P}(B^\comp) \geq 1-\delta,    
    \end{equation*}
    with $\KL(a\Vert p)$ and $\KL(b\Vert p)$ as in \eqref{eq:KLap} and \eqref{eq:KLbp} respectively. Except that if $\KL(0\Vert p) < \frac{\log r - \Delta(\delta, p, n)}{n}$, then with probability $> 1 - \delta$, $Z = 0$.
 
    The theorem follows by widening this interval, to get a symmetric expression.
\end{proof}

\begin{remark}
    The term $\left[\log \frac{p}{1-p}\right]_+$ is not actually needed in the upper bounds in Lemma \ref{lemma} and  Theorem \ref{thm1}, and is included in the Theorem statement only for the sake of a symmetric and concise statement.
\end{remark}

\begin{remark}
The upper bound in Theorem \ref{thm1} does not rely on $X_i$ being independent, nor even identically distributed. In fact, as long as $X_i \sim \frac{1}{n}\textnormal{Bin}(n, p_i)$ marginally with $p_i \geq p$, the upper bound still holds.  This union-bound based upper bound can be viewed as a (one sided) concentration guarantee, thinking of each $X_i$ as an empirical average with mean $p$, and ensuring that all $n$ empirical measurements are close to their expectation (or rather, not much smaller than their expectation).  In a machine learning context, $X_i$ would correspond to the empirical error of predictor $i$, with population error $p$ (or in the more general non-identically-distributed case, with population error $p_i\geq p$).  Viewed this way, the upper bound in Theorem \ref{thm1} is a special case\footnote{Theorem \ref{thm1} is a special case where we take a discrete uniform prior over the $r$ predictors, and are concerned only with point-mass posteriors.} of the PAC-Bayes bound as presented by \citet[][Equation (4)]{mcallester2003simplified} following \citet{langford2001bounds}.  The upper bound is thus similar to, but slightly tighter than, common concentration guarantees based on a similar union-bound argument, but relying on the Hoeffding and Bernstein bounds on the Binomial CDF instead of the KL-based bound of Lemma \ref{lemma}.  Theorem \ref{thm1} shows the KL-based bound is tight, and provides a simple matching lower bound.


\end{remark}


\bibliographystyle{plainnat}
\bibliography{reference}

\begin{thebibliography}{3}
\providecommand{\natexlab}[1]{#1}
\providecommand{\url}[1]{\texttt{#1}}
\expandafter\ifx\csname urlstyle\endcsname\relax
  \providecommand{\doi}[1]{doi: #1}\else
  \providecommand{\doi}{doi: \begingroup \urlstyle{rm}\Url}\fi

\bibitem[Cover and Thomas(2006)]{TC}
Thomas~M. Cover and Joy~A. Thomas.
\newblock \emph{Elements of Information Theory (Wiley Series in Telecommunications and Signal Processing)}.
\newblock USA, 2006.

\bibitem[Langford and Seeger(2002)]{langford2001bounds}
John Langford and Matthias Seeger.
\newblock Bounds for averaging classifiers.
\newblock Technical Report CMU-CS-01-102, CMU Computer Science, 2002.

\bibitem[McAllester(2003)]{mcallester2003simplified}
David McAllester.
\newblock Simplified {PAC-Bayesian} margin bounds.
\newblock In \emph{Learning Theory and Kernel Machines: 16th Annual Conference on Learning Theory and 7th Kernel Workshop, COLT/Kernel 2003, USA, 2003. Proceedings}, pages 203--215, 2003.

\end{thebibliography}

\end{document}